\documentclass[11pt, reqno]{amsart}

\usepackage[T1]{fontenc}
\usepackage{lmodern}
\usepackage{amsmath, amssymb, amsthm, mathtools}
\usepackage[margin=1.2in]{geometry}
\usepackage{microtype, enumitem, graphicx, subcaption}

\usepackage[unicode=true, hidelinks]{hyperref}
\hypersetup{pdftitle={The exact number of solutions to the planar isotropic Lp dual Minkowski problem}, pdfauthor={Guanyu Tao}, bookmarksnumbered=true}

\numberwithin{equation}{section}
\newtheorem{theorem}{Theorem}[section]
\newtheorem{proposition}[theorem]{Proposition}
\newtheorem{lemma}[theorem]{Lemma}
\newtheorem{corollary}[theorem]{Corollary}
\theoremstyle{remark}

\newcommand{\R}{\mathbb R}
\newcommand{\N}{\mathbb N}
\newcommand{\sphere}{\mathbb S^1}
\newcommand{\dd}{\,\mathrm d}
\newcommand{\ceil}[1]{\left\lceil #1 \right\rceil}

\allowdisplaybreaks[2]
\title[Planar isotropic $L_p$ dual Minkowski problem]{The exact number of nonconstant positive solutions to the planar isotropic $L_p$ dual Minkowski problem}
\author{Guanyu Tao}
\address{
Institute of Mathematics,
Academy of Mathematics and Systems Science,
Chinese Academy of Sciences,
No.\ 55 Zhongguancun East Road,
Beijing 100190,
China
}
\email{taoguanyu@amss.ac.cn}

\begin{document}

\begin{abstract}
We determine the exact number, up to rotations, of nonconstant positive $C^2$ solutions to the planar isotropic $L_p$ dual Minkowski problem for every $(p, q) \in \R^2$. We obtain a new parametrization of the associated period integral, characterize the regions where the period is strictly increasing or strictly decreasing, and prove that in the remaining nonmonotone region it has a unique nondegenerate maximum. As a consequence, we have a complete classification.
\end{abstract}

\maketitle

\section{Introduction}

The Brunn--Minkowski theory is one of the central theories in convex geometry and studies geometric measures and geometric inequalities of convex bodies; see, for example, Schneider \cite{Sch14}. The Minkowski problem is one of the basic problems in this theory and is also closely related to geometric analysis and fully nonlinear partial differential equations. Fundamental regularity results for the classical Minkowski problem were obtained by Cheng and Yau \cite{CY76}. Lutwak \cite{Lut93,Lut96} initiated the $L_p$ Brunn--Minkowski theory and introduced the $L_p$ Minkowski problem, which contains the classical Minkowski problem, the logarithmic Minkowski problem and the centro-affine Minkowski problem as important special cases. A landmark result for the logarithmic Minkowski problem was obtained by B\"or\"oczky, Lutwak, Yang and Zhang \cite{BLY13}, who solved the existence problem for even data. Many other important results on the $L_p$ Minkowski problem were obtained in \cite{BT17,CW06,HLY05,JLZ16,LW13,LYZ04}. 

On the dual side, Huang, Lutwak, Yang and Zhang \cite{HLY16} introduced the dual curvature measures, answering a longstanding question in the dual Brunn--Minkowski theory, and formulated the associated dual Minkowski problem. B\"or\"oczky, Lutwak, Yang, Zhang and Zhao \cite{BLY19} later gave a complete solution to the even dual Minkowski problem for $1<q<n$. Further developments of the dual Minkowski problem can be found in \cite{BHP18,CL18,HP18,Zhao17,Zhao18}. Subsequently, Lutwak, Yang and Zhang \cite{LYZ18} introduced the $L_p$ dual curvature measures and posed the $L_p$ dual Minkowski problem, which provides a unified framework containing both the $L_p$ Minkowski problem and the dual Minkowski problem as special cases. Huang and Zhao \cite{HZ18} established existence results in the weak sense and existence and uniqueness in the smooth category. Further results on the $L_p$ dual Minkowski problem, including smooth solutions and nonuniqueness, were obtained in \cite{BF19,CHZ19,CCL21,CL21,LLL22,SX21}.

In this paper, we study the planar isotropic $L_p$ dual Minkowski problem
\begin{equation} \label{eq:main}
    u^{1 - p} ((u')^2 + u^2)^\frac{q - 2}{2} (u'' + u) = 1, \qquad u > 0,
\end{equation}
on $\sphere$, where $(p, q) \in \R^2$. Since \eqref{eq:main} implies $u'' + u > 0$, every positive solution is the support function of a strictly convex body in $\R^2$ containing the origin in its interior. When $q = 2$, \eqref{eq:main} is the planar isotropic $L_p$ Minkowski problem. When $p = 0$, it is the planar isotropic dual Minkowski problem.

The study of the planar isotropic problem has a long history. A pioneering work was given by Andrews \cite{Ben03}, who reduced the classification of solutions of the planar isotropic $L_p$ Minkowski problem to the analysis of a period integral and obtained a complete classification. His work laid the foundation for the period integral method in the study of planar Minkowski-type problems. For the planar isotropic dual Minkowski problem, Liu and Lu \cite{LL26} made an important advance, where a main difficulty comes from the nonmonotonicity of the period integral. By combining theoretical analysis with numerical estimation of an integral with parameters, they developed an effective method to deal with this difficulty, determined the number of solutions for $0<q\leq 4$, and obtained an improved nonuniqueness result for $q>4$. More recently, Li and Wan \cite{LW26} carried out a systematic study of the general planar isotropic $L_p$ dual Minkowski problem. They established the endpoint asymptotic behavior and duality of the period integral, proved its monotonicity with respect to the parameters $p$ and $q$, and obtained monotonicity with respect to $r$ in large regions of the $(p,q)$-plane. They also obtained exact classification results in many regions and useful results in several nonmonotone regions. Their work greatly extended the period integral method to the general two-parameter problem and provides an important foundation for the present paper.

Building on the framework developed in the previous works, we obtain a complete classification for the planar isotropic $L_p$ dual Minkowski problem. More precisely, for every $(p,q)\in\mathbb{R}^2$, we determine the exact number, up to rotations, of nonconstant positive solutions, including the parameter regions where only lower bounds or partial classification results were previously known. We now state our main theorem.

\begin{theorem} \label{thm:main}
    Let $N(p, q)$ denote the number of $C^2$ nonconstant solutions to \eqref{eq:main} up to a rotation. Then except at $(1, 2), (-2, -1)$ and $(-2, 2)$, in which $N(p, q) = \infty$, one has
    \begin{enumerate}
        \item if $q - p \le 0$, then $N(p, q) = 0$;
        \item if $0 < q - p \le 1$, then
              \[
              N(p, q) =
              \begin{dcases}
                1, & q - p < 1, q > 2p, 2q > p, \\
                0, & \text{otherwise};
              \end{dcases}
              \]
        \item if $1 < q - p \le 4$, then
              \[
              N(p, q) =
              \begin{dcases}
                1, & q < 2p \text{ or } 2q < p, \\
                0, & \text{otherwise};
              \end{dcases}
              \]
        \item if $q - p > 4$, then
              \[
              N(p, q) =
              \begin{dcases}
                \ceil{\sqrt{q - p}} - 1, & q < 2p \text{ or } 2q < p, \\
                \ceil{\sqrt{q - p}} - 3, & p < 0 < q, \frac{1}{-p} + \frac{1}{q} \le 1, \\
                \ceil{\sqrt{q - p}} - 2, & \text{otherwise},
              \end{dcases}
              \]
    \end{enumerate}
    where $\ceil{\cdot}$ is the ceiling function.
\end{theorem}

Figure~\ref{fig:main} follows the same format as in Li--Wan\cite{LW26} for ease of comparison.

\begin{figure}[htbp]
    \centering
    \includegraphics[width=.8\textwidth]{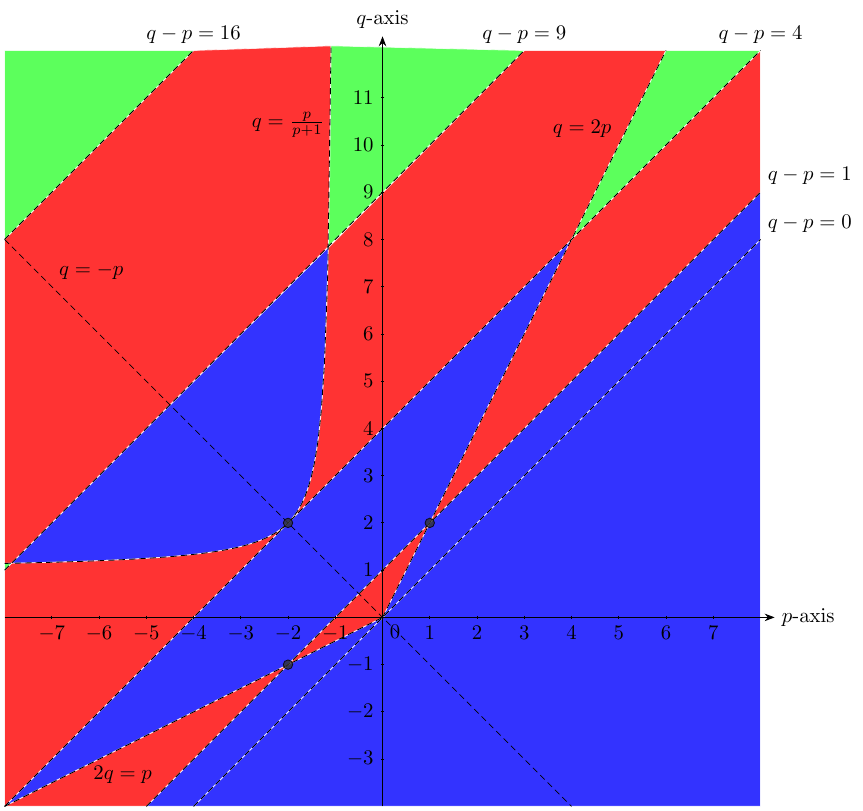}
    \caption{Blue domain: $N(p, q) = 0$. Red domain: $N(p, q) = 1$. Green domain: $N(p, q) = 2$.}
    \label{fig:main}
\end{figure}

The starting point of our argument is a new representation of the period integral. In the region $p \le 0 < q$, we introduce a parameter $E$ and a strictly convex function $U$. If $W$ denotes the distance between the two inverse branches of $U$, then the period integral admits a convolution representation
\[
\Theta(E) = \frac{1}{q - p} (K \ast W')(E).
\]
Moreover, the new parameter $E$ is strictly increasing with respect to the original parameter $r$. This representation allows us to study the global shape of $\Theta$ by analytic methods that are not apparent from the original period integral.

We then analyze the three possible behaviors of $\Theta(E)$. The strictly increasing region is characterized by studying the Laplace transform of the convolution formula. For the strictly decreasing region, we derive the differential identities of $W'$. In the remaining region, a further differential identity shows that $\Theta$ has at most one critical point, which, if it occurs, is a nondegenerate maximum. The corresponding results in the other regions are obtained by duality. Once the global profile of $\Theta$ is determined, the endpoint limits and a final upper bound yield the exact numbers of nonconstant solutions by counting the intersections of $\Theta$ with the levels $\pi/k$.

In Section~\ref{sec:2}, we recall the original period integral $\Theta(p, q, r)$ and basic properties established in Li-Wan in Subsection~\ref{sec:2.1}. We then introduce the new parameter $E$ and the convolution representation in Subsection~\ref{sec:2.2}, together with the differential identities for $W'$ that will be used later in Subsection~\ref{sec:2.3}. In Section~\ref{sec:3}, we determine the global behavior of $\Theta(E)$: Subsections~\ref{sec:3.1} and \ref{sec:3.2} characterize the strictly increasing and strictly decreasing regions, respectively, while Subsection~\ref{sec:3.3} treats the remaining region and proves the one-critical-point property. In Section~\ref{sec:4}, we prove Theorem~\ref{thm:main}.

\section{Preliminaries} \label{sec:2}

\subsection{The period integral} \label{sec:2.1}
Equation \eqref{eq:main} is equivalent to
\begin{equation} \label{eq:ode}
    u'' = u^{p - 1} ((u')^2 + u^2)^\frac{2 - q}{2} - u.
\end{equation}
The right-hand side, viewed as a function of $(u, u')$, is smooth on $(0, \infty) \times \R$. Hence the initial-value problem has a unique local solution, and every positive $C^2$ solution is smooth.

We first consider $p \ge q$. If $p > q$, at a global maximum and a global minimum,
\[
\left( \max_{\sphere} u \right)^{p - q} \le 1 \le \left( \min_{\sphere} u \right)^{p - q}.
\]
It follows that $u \equiv 1$. If $p = q$, the initial data $(u, u') = (c, 0)$ at a critical point are also satisfied by the constant solution $u \equiv c$. Uniqueness of the initial-value problem implies that $u$ is constant.

We therefore assume $q > p$ in the rest of this paper. Given a nonconstant solution $u$, at a critical point of $u$, equation \eqref{eq:ode} becomes
\begin{equation} \label{eq:critical}
    u'' = u(u^{-(q - p)} - 1).
\end{equation}
Note that $u$ cannot pass through $(u, u') = (1, 0)$, by uniqueness. Thus every critical point is nondegenerate. For a local minimum $u_-$ and a local maximum $u_+$, we have $u_- < 1 < u_+$. For $x > 0$, define
\[
\Phi_s(x) := \int_{1}^{x} t^{s - 1} \dd t =
\begin{dcases}
    \frac{x^s - 1}{s}, & s \neq 0 , \\
    \log x, & s = 0.
\end{dcases}
\]
Note that $\Phi_s$ is continuous and strictly increasing in $s$ for $x \neq 1$. By \eqref{eq:main}, we have a first integral
\[
\frac{\dd}{\dd \theta} \left[ \Phi_q(((u')^2 + u^2)^\frac{1}{2}) - \Phi_p(u) \right] = ((u')^2 + u^2)^\frac{q - 2}{2} u'(u'' + u) - u^{p - 1}u' = 0.
\]
Define $\Psi(x) := \Phi_q(x) - \Phi_p(x)$. Noting that $\Psi' < 0$ on $(0, 1)$ and $\Psi' > 0$ on $(1, \infty)$, every local minimum of $u$ has the same value and every local maximum also has the same value. Hence $u_- = \min_{\sphere} u$ and $u_+ = \max_{\sphere} u$. Let $r = u_+/u_- > 1$. Evaluating the first integral at $u_-$ and $u_+$ gives
\begin{equation} \label{eq:u_-}
    u_-^{q - p} = \frac{\Phi_p(r)}{\Phi_q(r)}.
\end{equation}
Let $x = u/u_- \in [1, r]$ and define
\[
I(x) :=
\begin{dcases}
    \left( 1 + (r^q - 1) \frac{\Phi_p(x)}{\Phi_p(r)} \right)^\frac{2}{q} - x^2, & q \neq 0, \\
    r^\frac{2\Phi_p(x)}{\Phi_p(r)} - x^2, & q = 0.
\end{dcases}
\]
Clearly $I(1) = I(r) = 0$. We claim that $I > 0$ on $(1, r)$. Indeed, for $\lambda \in [0, 1]$, define
\[
J_s(\lambda) :=
\begin{dcases}
    (1 + (r^s - 1)\lambda)^\frac{1}{s}, & s \neq 0, \\
    r^\lambda, & s = 0.
\end{dcases}
\]
Note that $J_s$ is continuous and strictly increasing in $s$. Taking $\lambda_0 = \Phi_p(x)/\Phi_p(r) \in [0, 1]$, we have $J_p(\lambda_0) = x$. Since $q > p$, it follows that $J_q(\lambda_0) > x$, which is exactly $I(x) > 0$.

Using the identity
\[
\Phi_s(tx) - \Phi_s(t) = t^s \Phi_s(x), \qquad t > 0,
\]
the first integral, by \eqref{eq:u_-}, can be written as
\begin{equation} \label{eq:u'}
    (u')^2 = u_-^2 I \left( \frac{u}{u_-} \right).
\end{equation}
Thus assume $u(0) = \min_{\sphere} u$ and the angle $\Theta$ between a minimum and the adjacent maximum is
\[
\Theta(p, q, r) = \int_{0}^{\Theta} \dd \theta = \int_{u_-}^{u_+} \frac{\dd u}{u'} = \int_{1}^{r} \frac{\dd x}{\sqrt{I(x)}}.    
\]
More explicitly,
\[
\Theta(p, q, r) =
\begin{dcases}
    \int_{1}^{r} \frac{\dd x}{\sqrt{\left( 1 + \frac{r^q - 1}{r^p - 1}(x^p - 1) \right)^\frac{2}{q} - x^2}}, & pq \neq 0, \\
    \int_{1}^{r} \frac{\dd x}{\sqrt{\left( 1 + \frac{r^q - 1}{\log r} \log x \right)^\frac{2}{q} - x^2}}, & p = 0, \\
    \int_{1}^{r} \frac{\dd x}{\sqrt{r^\frac{2(x^p - 1)}{r^p - 1} - x^2}}, & q = 0.
\end{dcases}
\]
Note that the integral $\Theta(p, q, r)$ is finite. Since \eqref{eq:ode} is autonomous and even in $u'$, uniqueness of the initial-value problem gives reflection symmetry about every critical point. Namely, for any critical point $\theta_0$, $u(\theta_0 + \theta) = u(\theta_0 - \theta)$. It follows that $u$ is $2\Theta$-periodic. Since $u \in C^2(\sphere)$, there must be a $k \in \N$ such that $\Theta(p, q, r) = \pi/k$.

On the other hand, given $r > 1$ and $\Theta(p, q, r) = \pi/k$ for some $k \in \N$, relation \eqref{eq:u_-} uniquely determines $u_- > 0$ and we set $u_+ = ru_-$. Consider the unique local solution of \eqref{eq:ode} with initial data $(u, u') = (u_-, 0)$ at $0$. Since $q > p$, by monotonicity of $\Phi_s$, we have $u_- < 1$. Hence by \eqref{eq:critical} $u''(0) > 0$, and then $u$ is increasing for short time. Note that the first integral and \eqref{eq:u'} remain valid. Since $I > 0$ on $(1, r)$, $u$ is strictly increasing as long as $u_- < u < u_+$. Then $u$ reaches $(u, u') = (u_+, 0)$ at time $\Theta$. By continuation, $u$ is well-defined on $[0, \Theta]$. Again since \eqref{eq:ode} is autonomous and even in $u'$, uniqueness gives reflection symmetry about $0$ and $\Theta$. It follows that $u \in C^2(\sphere)$ is a nonconstant solution with period $2\Theta$. 

The above discussion leads to the following proposition.

\begin{proposition}
    Let $N(p, q)$ be defined as in theorem~\ref{thm:main}, then
    \[
    N(p, q) = \#\left\{ r > 1: \Theta(p, q, r) = \frac{\pi}{k}, k \in \N \right\},
    \]
    where $\#$ denotes the number of elements.
\end{proposition}
Consequently, the analysis of $N(p, q)$ reduces entirely to that of $\Theta(p, q, r)$. We collect here several known analytical properties of $\Theta(p, q, r)$ established in the work of Li--Wan\cite{LW26}.

\begin{lemma}[{\cite[Theorem~5.1]{LW26}}] \label{lem:pqmonotone}
    For $q > p$ and fixed $r$, we have $\Theta(p, q, r)$ is strictly increasing in $p$ and strictly decreasing in $q$.
\end{lemma}

\begin{lemma}[{\cite[Theorem~3.1]{LW26}}] \label{lem:limit}
    For $q > p$, one has
    \begin{align}
        \lim_{r \to 1} \Theta(p, q, r) & = \frac{\pi}{\sqrt{q - p}}, \label{eq:limit1} \\
        \lim_{r \to \infty} \Theta(p, q, r) & =
        \begin{dcases}
            \frac{q}{q - p} \frac{\pi}{2}, & 0 \le p < q, \\
            \frac{\pi}{2}, & p < 0 < q, \\
            \frac{-p}{q - p} \frac{\pi}{2}, & p < q \le 0.
        \end{dcases} \label{eq:limit2}
    \end{align}
\end{lemma}

\begin{lemma}[{\cite[Theorem~4.1]{LW26}}] \label{lem:duality}
    For $q > p$, one has
    \begin{align}
        \Theta \left( -\frac{pq}{q - p}, q, r^\frac{q - p}{q} \right) & = \frac{q - p}{q} \Theta(p, q, r), & q > 0, \label{eq:dual1} \\
        \Theta \left(p, \frac{pq}{q - p}, r^\frac{q - p}{-p} \right) & = \frac{q - p}{-p} \Theta(p, q, r), & p < 0, \label{eq:dual2} \\
        \Theta(-q, -p, r) & = \Theta(p, q, r). \label{eq:dual3}
    \end{align}
\end{lemma}

\subsection{The new representation} \label{sec:2.2}

We now derive the new representation of $\Theta$. In the subsequent analysis, we only consider the region $p \le 0 < q$, while the results for the remaining regions follow immediately by duality~\eqref{eq:dual1} and \eqref{eq:dual2}.

For a nonconstant solution $u$, fix $u(0) = u_-$, then $u(\Theta) = u_+$ and $u' > 0$ on $(0, \Theta)$. Define
\begin{equation} \label{eq:eta}
    \eta := \theta + \arctan \frac{u'}{u}.
\end{equation}
Geometrically, this represents the polar angle of the position vector of the convex body with support function $u$, or equivalently, the normal angle of its dual body. By \eqref{eq:main}, a direct calculation gives $\eta' = u^p((u')^2 + u^2)^{-q/2} > 0$. For convenience, define
\[
y := \log \eta' = p \log u - \frac{q}{2} \log((u')^2 + u^2).
\]
Note that $qe^y - p > 0$ always holds under $p \le 0 < q$. Then
\begin{equation} \label{eq:y'}
    y' = -(qe^y - p) \frac{u'}{u} < 0
\end{equation}
on $(0, \Theta)$. Define
\[
U(y) := \int_{0}^{y} \frac{e^t - 1}{qe^t - p} \dd t =
\begin{dcases}
    \frac{y}{p} - \frac{q - p}{qp} \log \frac{qe^y - p}{q - p}, & p \neq 0, \\
    \frac{y + e^{-y} - 1}{q}, & p = 0.
\end{dcases}
\]
By direct computation, we obtain
\[
\frac{\dd}{\dd \theta} \left[ \frac{1}{2} \log \left( 1 + \frac{(u')^2}{u^2} \right) + U(y) \right] = 0.
\]
Thus there exists a constant $E$ such that
\begin{equation} \label{eq:energy}
    \frac{1}{2} \log \left( 1 + \frac{(u')^2}{u^2} \right) + U(y) = E.
\end{equation}
Here, $E$ serves as our new parameter. In what follows, we will parameterize $\Theta$ in terms of $E$ instead of the original $r$. Under this choice of parameter, the integral exhibits a much more favorable structure. Note that
\[
U(0) = U'(0) = 0, \quad U''(y) = \frac{(q - p)e^y}{(qe^y - p)^2} > 0, \quad U(y) \to \infty \text{ as } y \to \pm \infty.
\]
Thus $U$ is strictly convex and $y = 0$ is the unique minimum point. Hence for every $x > 0$, $U(y) = x$ has two branches $y_L(x) < 0 < y_R(x)$. Define
\[
W(x) := y_R(x) - y_L(x).
\]
Clearly $W' > 0$. Note that
\begin{equation} \label{eq:W(E)}
    W(E) = y_R(E) - y_L(E) = -(q - p) \log u_- + (q - p) \log u_+ = (q - p) \log r.
\end{equation}
It follows that $E$ maps $r \in (1, \infty)$ increasingly onto $E \in (0, \infty)$. At a critical point $u' = 0$, by \eqref{eq:eta}, we have $\eta = \theta$. Then by \eqref{eq:y'} and \eqref{eq:energy}, we obtain
\begin{align*}
    \Theta & = \frac{1}{q - p} \int_{0}^{\Theta} (q \dd \eta - p \dd \theta) = \frac{1}{q - p} \int_{0}^{\Theta} (qe^y - p) \dd \theta \\
    & = \frac{1}{q - p} \int_{0}^{\Theta} -\frac{y'}{u'/u} \dd \theta = \frac{1}{q - p} \int_{y_L(E)}^{y_R(E)} \frac{\dd y}{\sqrt{e^{2(E - U(y))} - 1}} \\
    & = \frac{1}{q - p} \int_{0}^{E} \frac{W'(x) \dd x}{\sqrt{e^{2(E - x)} - 1}}.
\end{align*}
For brevity, we write $K(t) = (e^{2t} - 1)^{-1/2}$ and for fixed $(p, q)$, we arrive at a convolution representation
\begin{equation} \label{eq:Theta(E)}
    \Theta(E) := \frac{1}{q - p} (K \ast W')(E).
\end{equation}
In view of relation \eqref{eq:W(E)}, studying the monotonicity with respect to $E$ intuitively translates into that with respect to $r$.

\subsection{The analysis of the width} \label{sec:2.3}
To study $\Theta(E)$, we now turn to the properties of $W'(x)$. Note that
\[
U'' = \frac{1}{q - p} (1 - pU')(1 - qU').
\]
Differentiating $U(y_i(x)) = x$, we obtain
\begin{equation} \label{eq:y_i-ode}
    y_i'' = -\frac{1}{q - p} y_i'(y_i' - p)(y_i' - q), \quad i = L, R.
\end{equation}
Note that for $x > 0$, we have $y_L'(x) < p \le 0 < q < y_R'(x)$. Thus $y_R''(x) < 0 < y_L''(x)$. Recall that $W' = y_R' - y_L'$. For the subsequent analysis, we introduce $H := y_R' + y_L'$. Then by direct computation from \eqref{eq:y_i-ode}, one has
\[
W'' = -\frac{(W')^3 + CW'}{4(q - p)},
\]
where
\begin{equation} \label{eq:C}
    C(x) := 3\left( H(x) - \frac{2(q + p)}{3} \right)^2 - \frac{(q + p)^2}{3} - (q - p)^2.
\end{equation}
In order to transform the ODE into a linear equation, we introduce $Y := (W')^{-2}$. Thus one has
\begin{equation} \label{eq:Y-ode}
    Y' = \frac{CY + 1}{2(q - p)}.
\end{equation}
Next, we need to analyze $H(x)$, mainly focusing on two properties. The first is that
\begin{equation} \label{eq:|H'|}
    |H'| = |y_R'' + y_L''| < -y_R'' + y_L'' = -W'' = \frac{Y'}{2Y^{3/2}}.
\end{equation}
The second is the following lemma.

\begin{lemma} \label{lem:H}
    For $p \le 0 < q$, if $q + p = 0$, then $H \equiv 0$; if $q + p \neq 0$, then
    \[
    \left( H - \frac{2(q + p)}{3} \right)H' > 0.
    \]
\end{lemma}

\begin{proof}
    If $q + p = 0$, then
    \[
    U'(y) = \frac{1}{q} \tanh \frac{y}{2}, \quad U(y) = \frac{2}{q} \log \cosh \frac{y}{2}.
    \]
    Since $U(y)$ is even, then $y_L(x) = -y_R(x)$. Hence $H(x) = y_R'(x) + y_L'(x) = 0$.

    If $q + p \neq 0$, by direct calculation from \eqref{eq:y_i-ode} and \eqref{eq:Y-ode}, we obtain
    \[
    H' = -\frac{1}{4(q - p)} \left[ \frac{3H - 2(q + p)}{Y} + H[(H - (q + p))^2 - (q - p)^2] \right],
    \]
    and
    \begin{equation} \label{eq:H-ode}
        H'' = -\frac{1 + (q - p)Y'}{2(q - p)Y}H' + \frac{(3H - 2(q + p))Y'}{4(q - p)Y^2}.
    \end{equation}
    Note that
    \[
    H(0+) = \frac{2(q + p)}{3}, \quad H'(0+) = \frac{(q + p)[9(q - p)^2 - (q + p)^2]}{135(q - p)}.
    \]
    At a critical point $H' = 0$, equation \eqref{eq:H-ode} becomes
    \[
    H'' = \frac{(3H - 2(q + p))Y'}{4(q - p)Y^2}.
    \]
    By \eqref{eq:|H'|}, $Y' > 0$. If $q + p > 0$, then $H'(0+) > 0$. We claim that $H' > 0$ on $(0, \infty)$. Otherwise, if $H'$ has a first zero $x_0 > 0$, then $3H(x_0) - 2(q + p) > 0$. Hence $H''(x_0) > 0$ is a local minimum, which is a contradiction. Therefore $3H - 2(q + p) > 0$ and this lemma holds. The argument in the case $q + p < 0$ is similar. 
\end{proof}

This immediately gives

\begin{corollary} \label{cor:C}
    For $p \le 0 < q$, if $q + p = 0$, then $C$ is constant; if $q + p \neq 0$, then
    \begin{equation} \label{eq:C'}
        C' = 6 \left( H - \frac{2(q + p)}{3} \right) H' > 0.
    \end{equation}
\end{corollary}

At the end of this section, we present the final key property of $Y$.

\begin{lemma} \label{lem:log-convexity}
    For $p \le 0 < q$ and $q + p \neq 0$, $Y'$ is strictly log-convex. Namely, 
    \[
    (\log Y')'' > 0.
    \]
    Moreover,
    \[
    -(\log Y')' = -\frac{Y''}{Y'} > 0.
    \]
\end{lemma}

\begin{proof}
    Differentiating \eqref{eq:C'}, we have
    \begin{equation} \label{eq:C''}
        C'' = 6(H')^2 + (6H - 4(q + p))H''.
    \end{equation}
    Differentiating \eqref{eq:Y-ode}, we have
    \begin{equation}
        Y'' = \frac{C'Y + CY'}{2(q - p)}.
    \end{equation}
    Differentiating again, we obtain
    \begin{equation}
        Y''' = \frac{C''Y + 2C'Y' + CY''}{2(q - p)}.
    \end{equation}
    Note that
    \[
    (\log Y')'' = \frac{Y'''}{Y'} - \left( \frac{Y''}{Y'} \right)^2.
    \]
    Using \eqref{eq:C}, \eqref{eq:H-ode}, \eqref{eq:C'}, \eqref{eq:C''} and eliminating $H'', C, C', C''$, we obtain
    \[
    (\log Y')'' = \frac{3Y(H')^2}{(q - p)Y'} + \frac{(3H - 2(q + p))H'}{2(q - p)} + \frac{(3H - 2(q + p))^2}{4(q - p)^2Y} \left( 1 - \frac{4Y^3(H')^2}{(Y')^2} \right).
    \]
    By \eqref{eq:|H'|} and lemma~\ref{lem:H}, we have the three terms are strictly positive. This gives the log-convexity. 
    
    Moreover, $-Y''/Y'$ is strictly decreasing. Suppose $-Y''/Y' \le 0$ at some point. Then there exist $x_0 > 0$ and $c > 0$ such that
    \[
    -\frac{Y''}{Y'} \le -c, \quad x \ge x_0.
    \]
    Thus $Y'$ grows at least exponentially, and consequently $Y(x) \to \infty$ as $x \to \infty$. However, one has $W' > q - p > 0$. Then
    \[
    0 < Y(x) < \frac{1}{(q - p)^2}.
    \]
    This is a contradiction.
\end{proof}

\section{Monotonicity} \label{sec:3}

\subsection{The increasing regions} \label{sec:3.1}

We will identify the region in the $(p, q)$-plane where $\Theta$ is strictly increasing with respect to $E$, or equivalently, with respect to $r$, in view of \eqref{eq:W(E)}. Before proceeding, we establish a lemma. We will subsequently show that the region where the following function is strictly positive coincides with the region where $\Theta'(E) > 0$.

\begin{lemma} \label{lem:F}
    For $p < 0 < q$, and set
    \[
    \omega := -\frac{pq}{q - p} > 0, \quad \text{or equivalently,} \quad \frac{1}{\omega} = \frac{1}{-p} + \frac{1}{q}.
    \]
    Define
    \[
    F(t) := \frac{1}{e^t + 1} + \frac{1}{e^{-pt} - 1} + \frac{1}{e^{qt} - 1} - \frac{1}{e^{\omega t} - 1}, \quad t > 0.
    \]
    Then
    \[
    F(t) \ge 0, \quad t > 0 \quad \text{if and only if} \quad \omega \ge 1.
    \]
    Moreover, if $\omega \ge 1$, then $F(t) > 0$ for every $t > 0$, except at $(p, q) = (-2, 2)$, where $F \equiv 0$.
\end{lemma}

\begin{proof}
    Rewrite
    \[
    F(t) = \left( \frac{1}{e^t + 1} - \frac{1}{e^{\omega t} + 1} \right) + \left( \frac{1}{e^{-pt} - 1} + \frac{1}{e^{qt} - 1} - \frac{2}{e^{2 \omega t} - 1} \right) =: \mathrm{I}(t) + \mathrm{II}(t).
    \]
    If $\omega \ge 1$, clearly $\mathrm{I}(t) \ge 0$ and equality holds iff $\omega = 1$. Define
    \[
    f(r) = \frac{1}{e^\frac{\omega t}{r} - 1}, \quad r > 0.
    \]
    A direct computation gives $f''(r) > 0$ and then $f$ is strictly convex. By Jensen's inequality
    \[
    \frac{1}{e^{-pt} - 1} + \frac{1}{e^{qt} - 1} = f \left( \frac{q}{q - p} \right) + f \left( \frac{-p}{q - p} \right) \ge 2 f \left( \frac{1}{2} \right) = \frac{2}{e^{2 \omega t} - 1},
    \]
    which means $\mathrm{II}(t) \ge 0$ and equality holds iff $q = -p$. Consequently, $F(t) > 0$ unless $(p, q) = (-2, 2)$.

    Conversely, if $\omega < 1$, one has $\omega < \min \{ -p, q, 1 \}$. Hence $e^{\omega t}F(t) \to -1$ as $t \to \infty$. Then $F(t) < 0$ for $t > 0$ sufficiently large.
\end{proof}

We now state the theorem characterizing the strictly increasing region. Figures~\ref{fig:LW-increasing} and \ref{fig:increasing} illustrate a comparison between the region established by Li--Wan\cite{LW26} and the region obtained in our work.

\begin{theorem}
    Suppose $q > p$ and $(p, q) \neq (1, 2), (-2, -1), (-2, 2)$. Then $\Theta(p, q, r)$ is strictly increasing in $r$ if and only if
    \[
    p < 0 < q \text{ and } \frac{1}{-p} + \frac{1}{q} \le 1 \quad \text{or} \quad p \ge 1 \quad \text{ or } \quad q \le -1.
    \]
\end{theorem}

\begin{figure}[htbp]
    \centering
    \begin{subfigure}[b]{0.48\textwidth}
        \centering
        \includegraphics[width=\linewidth]{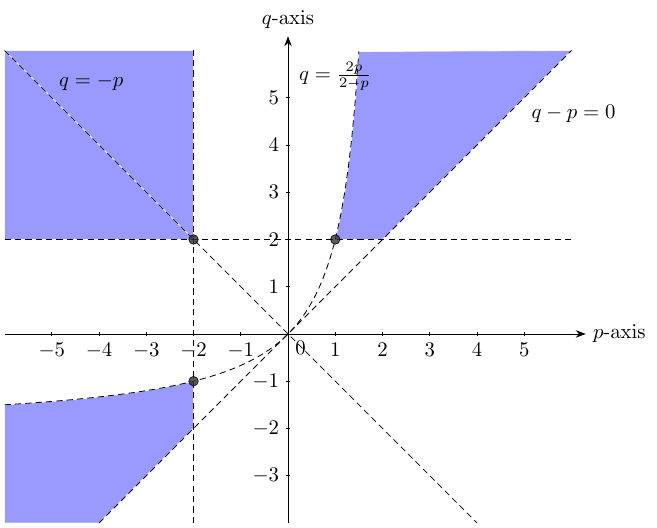}
        \caption{The increasing region (blue) in Li--Wan.}
        \label{fig:LW-increasing}
    \end{subfigure}
    \hfill
    \begin{subfigure}[b]{0.48\textwidth}
        \centering
        \includegraphics[width=\linewidth]{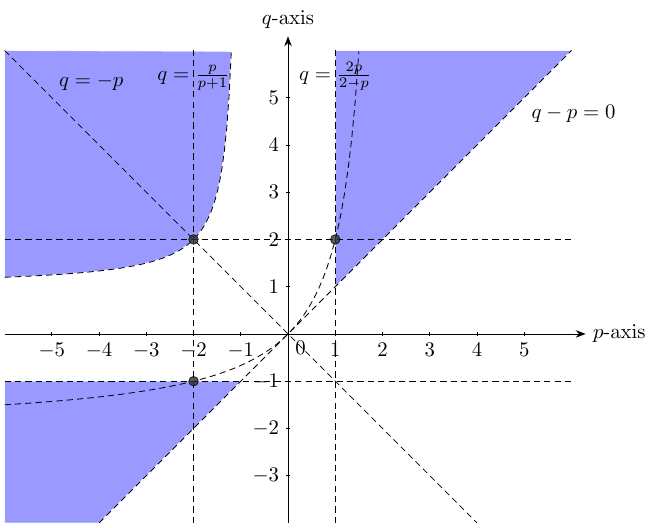}
        \caption{The increasing region (blue) in our work.}
        \label{fig:increasing}
    \end{subfigure}
    \caption{Comparison between the strictly increasing regions}
\end{figure}

\begin{proof}
    We first prove the sufficiency. By duality~\eqref{eq:dual1} and \eqref{eq:dual2}, it suffices to consider $p \le 0 < q$ and $\omega \ge 1$. We first assume $p < 0 < q$. For a function $f$ defined on $(0, \infty)$, write
    \[
    \hat{f}(s) := \int_{0}^{\infty} e^{-st}f(t) \dd t, \quad s > 0. 
    \]
    The Laplace transform of $K$ and $W'$ are well-defined for $s > 0$. From the convolution \eqref{eq:Theta(E)},
    \[
    \hat{\Theta}(s) = \frac{1}{q - p} \hat{K}(s) \widehat{W'}(s).
    \]
    We directly compute
    \[
    \hat{K}(s) = \int_{0}^{\infty} \frac{e^{-st}}{\sqrt{e^{2t} - 1}} \dd t = \frac{1}{2} B \left( \frac{s + 1}{2}, \frac{1}{2} \right),
    \]
    and
    \[
    \widehat{W'}(s) = \int_{-\infty}^{\infty} e^{-sU(y)} \dd y = \int_{-\infty}^{\infty} e^{\frac{y}{-p}s} \left( \frac{qe^y - p}{q - p} \right)^{\frac{q - p}{qp}s} \dd y = C^s B \left( \frac{s}{-p}, \frac{s}{q} \right),
    \]
    where $C = C(p, q) > 0$ is independent of $s$. Using $B(a, b) = \Gamma(a) \Gamma(b) / \Gamma(a + b)$, we obtain
    \begin{equation} \label{eq:sTheta}
        s\hat{\Theta}(s) = \frac{\sqrt{\pi}}{q - p} C^s \frac{\Gamma \left( \frac{s + 1}{2} \right) \Gamma \left( \frac{s}{-p} \right) \Gamma \left( \frac{s}{q} \right)}{\Gamma \left( \frac{s}{2} \right) \Gamma \left( \frac{s}{\omega} \right)}.
    \end{equation}
    We next recover the sign of $\Theta'$ from this formula. Recall the standard identity
    \[
    (\log \Gamma)''(z) = \int_{0}^{\infty} \frac{te^{-zt}}{1 - e^{-t}} \dd t, \quad z > 0.
    \]
    Differentiating the logarithm of \eqref{eq:sTheta} twice gives
    \[
    \frac{\dd^2}{\dd s^2} \log(s \hat{\Theta}(s)) = \int_{0}^{\infty} te^{-st} F(t) \dd t,
    \]
    where $F$ is defined as in lemma~\ref{lem:F}. Set $g(t) := F(t)/t$. The behavior of $F$ at $0$ and at infinity shows that $g$ is bounded and integrable. Since
    \[
    \hat{g}''(s) = \int_{0}^{\infty} te^{-st} F(t) \dd t,
    \]
    we have $\log(s\hat{\Theta}(s)) - \hat{g}(s)$ is affine in $s$. Note that $s\hat{\Theta}(s) \to \Theta(0) = \pi/\sqrt{q - p}$ as $s \to \infty$, while $\hat{g}(s) \to 0$. Hence the affine function above is constant, and therefore
    \[
    s\hat{\Theta}(s) = \frac{\pi}{\sqrt{q - p}} e^{\hat{g}(s)} = \frac{\pi}{\sqrt{q - p}} \left( 1 + \sum_{n = 1}^{\infty} \frac{\hat{g}^n(s)}{n!} \right).
    \]
    Since $g \in L^1 \cap L^\infty$, the resulting series converges absolutely. Thus
    \begin{equation} \label{eq:Theta'(E)}
        \Theta'(E) = \frac{\pi}{\sqrt{q - p}} \sum_{n = 1}^{\infty} \frac{g^{\ast n}(E)}{n!}.
    \end{equation}
    If $\omega \ge 1$, lemma~\ref{lem:F} gives $F(t) > 0$ for every $t > 0$, except at $(p, q) = (-2, 2)$. Hence $g > 0$ and every term in \eqref{eq:Theta'(E)} is positive. Therefore
    \[
    \Theta'(E) > 0, \quad E > 0.
    \]

    For necessity, we first consider $p < 0 < q$. If $\omega < 1$, using $\Gamma(z + 1) = z\Gamma(z)$, formula \eqref{eq:sTheta} can be rewritten as
    \[
    G(s) := s\hat{\Theta}(s) = \frac{\sqrt{\pi}}{2} C^s \frac{\Gamma \left( \frac{s + 1}{2} \right) \Gamma \left( 1 + \frac{s}{-p} \right) \Gamma \left( 1 + \frac{s}{q} \right)}{\Gamma \left( 1 + \frac{s}{2} \right) \Gamma \left( 1 + \frac{s}{\omega} \right)}.
    \]
    Although $G$ was initially defined only for $s > 0$, the right-hand side defines an analytic continuation of $G$ to $|s| < \min \{ 1, -p, q \}$. Since $\omega < \min \{ 1, -p, q \}$, at $s = -\omega$,
    \[
    \frac{1}{\Gamma(1 + s/\omega)} = \frac{1}{\Gamma(0)} = 0,
    \]
    while all other factors are finite and nonzero. Hence $G(-\omega) = 0$. On the other hand, since $\Theta' \ge 0$ and both endpoint limits of $\Theta$ are finite, $\Theta' \in L^1(0, \infty)$. For $s > 0$, integration by parts gives
    \[
    G(s) = \Theta(0+) + \int_{0}^{\infty} e^{-st} \Theta'(t) \dd t. 
    \]
    Consequently, for every $n \ge 1$,
    \[
    (-1)^n G^{(n)}(s) = \int_{0}^{\infty} t^n e^{-st} \Theta'(t) \dd t.
    \]
    Letting $s \to 0+$, monotone convergence together with the analyticity of $G$ at $0$ yields $(-1)^n G^{(n)}(0) > 0$. Since the Taylor series converges at $s = -\omega$,
    \[
    G(-\omega) = \sum_{n = 0}^{\infty} \frac{G^{(n)}(0)}{n!} (-\omega)^n = G(0) + \sum_{n = 1}^{\infty} \frac{(-1)^n G^{(n)}(0)}{n!} \omega^n > \Theta(\infty) = \frac{\pi}{2}. 
    \]
    This is a contradiction, therefore $\omega \ge 1$.
    
    It remains to exclude $p = 0 < q$. Recall that in this case
    \[
    U(y) = \frac{y + e^{-y} - 1}{q}.
    \]
    We directly compute
    \[
    \widehat{W'}(s) = \int_{-\infty}^{\infty} e^{-sU(y)} \dd y = e^\frac{s}{q} \left( \frac{q}{s} \right)^\frac{s}{q} \Gamma \left( \frac{s}{q} \right),
    \]
    and then
    \[
    G(s) := s\hat{\Theta}(s) =  \frac{\sqrt{\pi}}{2} \frac{\Gamma \left( \frac{s + 1}{2} \right) \Gamma \left( 1 + \frac{s}{q} \right)}{\Gamma \left( 1 + \frac{s}{2} \right)} e^\frac{s}{q} \left( \frac{q}{s} \right)^\frac{s}{q}.
    \]
    Note that all factors in the right-hand side, except $(q/s)^{s/q}$, are analytic and nonzero at $s = 0$. Hence
    \[
    \log \frac{G(s)}{\pi/2}  = -\frac{s}{q} \log s + O(s).
    \]
    It follows that $G(s) > \pi/2$ for $s > 0$ sufficiently small. On the other hand, since $\Theta(\infty) = \pi/2$, then $\Theta(E) < \pi/2$ for every $E > 0$. However, for every $s > 0$,
    \[
    G(s) = \int_{0}^{\infty} se^{-st}\Theta(t) \dd t < \frac{\pi}{2} \int_{0}^{\infty} se^{-st} \dd t = \frac{\pi}{2}.
    \]
    This is a contradiction, therefore $\Theta$ is not strictly increasing.
\end{proof}

\subsection{The decreasing regions} \label{sec:3.2}
Next we identify the strictly decreasing region in the $(p, q)$-plane. Figures~\ref{fig:LW-decreasing} and \ref{fig:decreasing} compare the strictly decreasing region established by Li--Wan\cite{LW26} with that obtained in our result.

\begin{theorem}
    Suppose $q > p$ and $(p, q) \neq (1, 2), (-2, -1), (-2, 2)$. Then $\Theta(p, q, r)$ is strictly decreasing in $r$ if and only if
    \[
    p^2 - pq + q^2 + 3p - 3q \le 0.
    \]
\end{theorem}

\begin{figure}[htbp]
    \centering
    \begin{subfigure}[b]{0.48\textwidth}
        \centering
        \includegraphics[width=\linewidth]{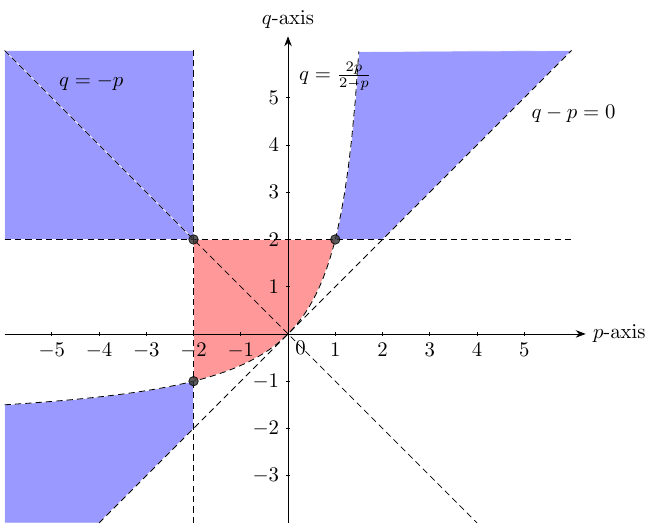}
        \caption{The decreasing region (red) in Li--Wan.}
        \label{fig:LW-decreasing}
    \end{subfigure}
    \hfill
    \begin{subfigure}[b]{0.48\textwidth}
        \centering
        \includegraphics[width=\linewidth]{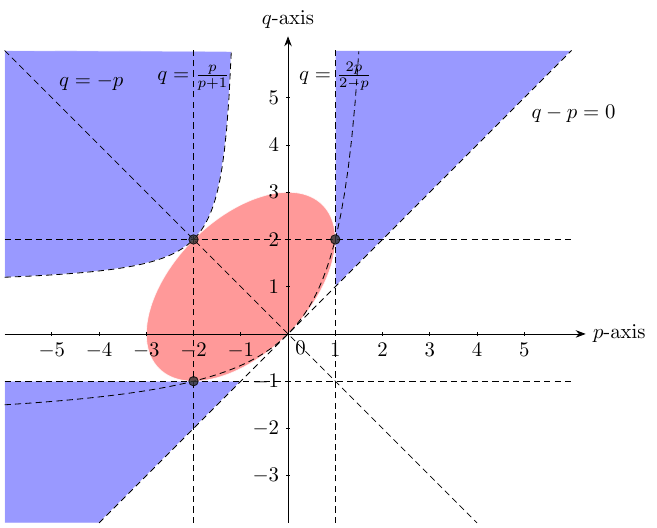}
        \caption{The decreasing region (red) in our work.}
        \label{fig:decreasing}
    \end{subfigure}
    \caption{Comparison between the strictly decreasing regions}
\end{figure}

\begin{proof}
    We first prove the sufficiency. By duality~\eqref{eq:dual1} and \eqref{eq:dual2}, it suffices to consider $p \le 0 < q$. First assume that $p + q \neq 0$. Noting that $H(x) \to q + p$, by corollary~\ref{cor:C}, we have
    \[
    -(q - p)^2 - \frac{(q + p)^2}{3} = C(0) < C(x) < -(q - p)^2 < 0, \quad x > 0.
    \]
    Solving the equation~\eqref{eq:Y-ode}, we obtain
    \begin{align*}
        Y(x) & = \frac{1}{2(q - p)} \int_{0}^{x} \exp \left\{ \frac{1}{2(q - p)} \int_{t}^{x} C(s) \dd s \right\} \dd t \\
        & < \frac{1}{2(q - p)} \int_{0}^{x} \exp \left\{ \frac{C(x)(x - t)}{2(q - p)} \right\} \dd t \\
        & = \frac{e^\frac{C(x)x}{2(q - p)} - 1}{C(x)}.
    \end{align*}
    Since $p^2 - pq + q^2 + 3p - 3q \le 0$, we have
    \[
    0 < -\frac{C(x)}{2(q - p)} < -\frac{C(0)}{2(q - p)} \le 2.
    \]
    Note that for fixed $x > 0$, the function $r \mapsto r/(e^{rx} - 1)$ is strictly decreasing on $(0, \infty)$. Using \eqref{eq:Y-ode} again yields
    \begin{equation} \label{eq:Y'/Y}
        \frac{Y'(x)}{Y(x)} > \frac{-\frac{C(x)}{2(q - p)}}{e^{-\frac{C(x)}{2(q - p)}x} - 1} > \frac{2}{e^{2x} - 1}.
    \end{equation}
    Define
    \[
    M(x) := W'(x) \sqrt{1 - e^{-2x}} = \sqrt{\frac{1 - e^{-2x}}{Y(x)}} > 0 \quad x > 0.
    \]
    Then by \eqref{eq:Y'/Y}
    \begin{equation} \label{eq:logM}
        (\log M)' = \frac{1}{e^{2x} - 1} - \frac{Y'}{2Y} < 0.
    \end{equation}
    Thus $M$ is strictly decreasing. Recall the convolution representation \eqref{eq:Theta(E)}, let
    \[
    z := \frac{e^{2x} - 1}{e^{2E} - 1} \in (0, 1), \quad \text{i.e.} \quad x = \frac{1}{2} \log(1 + (e^{2E} - 1)z),
    \]
    then
    \begin{equation} \label{eq:new Theta(E)}
        \Theta(E) = \frac{1}{2(q - p)} \int_{0}^{1} \frac{M(\frac{1}{2} \log(1 + (e^{2E} - 1)z))}{\sqrt{z(1 - z)}} \dd z.
    \end{equation}
    For every $z \in (0, 1)$, $\log(1 + (e^{2E} - 1)z)$ is strictly increasing in $E$. It follows that $\Theta'(E) < 0$.

    It remains to consider $q + p = 0$. In this case $H \equiv 0$ and $C \equiv C(0)$. The equation~\eqref{eq:Y-ode} gives explicitly
    \[
    Y(x) = \frac{1 - e^{-\frac{q - p}{2}x}}{(q - p)^2}.
    \]
    Consequently,
    \[
    M(x) = (q - p) \sqrt{\frac{1 - e^{-2x}}{1 - e^{-\frac{q - p}{2}x}}},
    \]
    and
    \[
    (\log M)' = \frac{1}{e^{2x} - 1} - \frac{\frac{q - p}{4}}{e^{\frac{q - p}{2}x} - 1}.
    \]
    The elliptic condition reduces to $q - p \le 4$. It follows that $M' < 0$ whenever $q - p < 4$. The equality $q - p = 4$ corresponds to $(p, q) = (-2, 2)$, which has been excluded.
    
    For necessity, we use the expansion of Li--Wan\cite[Corollary~4.3]{LW26}:
    \[
    \Theta(p, q, r) = \frac{\pi}{\sqrt{q - p}} + \frac{\pi(p^2 - pq + q^2 + 3p - 3q)}{96\sqrt{q - p}} (r - 1)^2 + o((r - 1)^2).
    \]
    If $p^2 - pq + q^2 + 3p - 3q > 0$, then $\Theta(p, q, r)$ will increase for $r > 1$ sufficiently close to $1$.
\end{proof}

\subsection{The non-monotonic regions} \label{sec:3.3}

In this subsection, we prove that for the region $q > p$, $\Theta(p, q, r)$ has at most one critical point with respect to $r$. Moreover, if such a critical point exists, it must be a strict local maximum. Combining this with the regions of monotonicity established earlier, it follows that $\Theta(p, q, r)$ has exactly one critical point, which is a strict maximum. Figures~\ref{fig:LW-nonmonotone} and \ref{fig:nonmonotone} illustrate the regions lacking monotonicity for which Li--Wan\cite{LW26} were able to determine the exact number of solutions.

\begin{theorem} \label{thm:nonmonotone}
    Suppose $q > p$, and $(p, q) \neq (1, 2), (-2, -1), (-2, 2)$. Then $\Theta(p, q, r)$ has at most one critical point on $r > 1$. Moreover, every critical point, if it exists, is a nondegenerate local maximum.
\end{theorem}

\begin{figure}[htbp]
    \centering
    \begin{subfigure}[b]{0.48\textwidth}
        \centering
        \includegraphics[width=\linewidth]{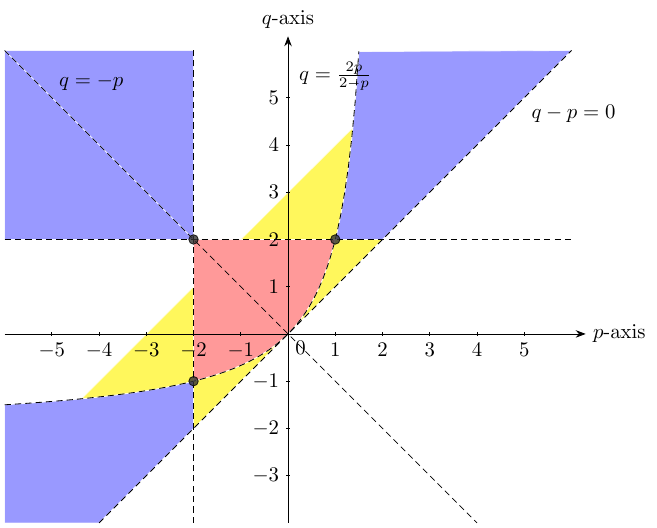}
        \caption{The non-monotonic regions (yellow) where Li--Wan determined the exact number.}
        \label{fig:LW-nonmonotone}
    \end{subfigure}
    \hfill
    \begin{subfigure}[b]{0.48\textwidth}
        \centering
        \includegraphics[width=\linewidth]{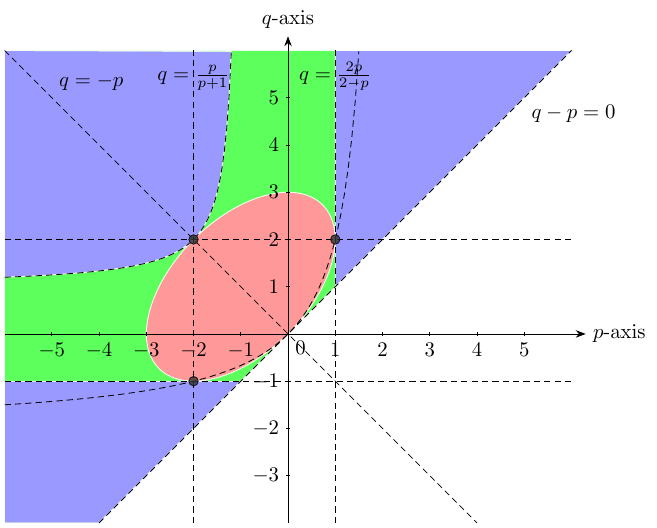}
        \caption{The regions (green) with exactly one critical point in our work.}
        \label{fig:nonmonotone}
    \end{subfigure}
    \caption{Results for nonmonotone regions}
\end{figure}

We first establish an integral lemma.

\begin{lemma} \label{lem:sign}
    Let $f$ be continuous on $(0, \infty)$ and integrable near $0$. Assume that for some $z_0 > 0$,
    \[
    f(z) > 0 \quad (0 < z < z_0), \quad f(z) < 0 \quad (z > z_0).
    \]
    Define
    \[
    A(h) := \int_{0}^{h} f(z) \sqrt{h - z} \dd z.
    \]
    Then every zero $h_\ast > 0$ of $A$ satisfies
    \[
    h_\ast > z_0, \quad A'(h_\ast) < 0.
    \]
\end{lemma}

\begin{proof}
    If $h \le z_0$, then $f(z) > 0$ on $(0, h)$. Hence $A(h) > 0$. This implies zero $h_\ast$ must satisfy
    \[
    h_\ast > z_0.
    \]
    Differentiating $A(h)$, we have
    \[
    A'(h) = \frac{1}{2} \int_{0}^{h} \frac{f(z)}{\sqrt{h - z}} \dd z.
    \]
    Then
    \begin{align*}
        A'(h_\ast) & = \frac{1}{2} \int_{0}^{h_\ast} f(z) \left[ \frac{1}{\sqrt{h_\ast - z}} - \frac{\sqrt{h_\ast - z}}{h_\ast - z_0} \right] \dd z \\
        & = \frac{1}{2} \int_{0}^{h_\ast} f(z) \sqrt{h_\ast - z} \left[ \frac{1}{h_\ast - z} - \frac{1}{h_\ast - z_0} \right] \dd z. 
    \end{align*}
    If $0 < z < z_0$, then $f(z) > 0$ but
    \[
    \frac{1}{h_\ast - z} - \frac{1}{h_\ast - z_0} < 0.
    \]
    If $z_0 < z < h_\ast$, then $f(z) < 0$ but
     \[
    \frac{1}{h_\ast - z} - \frac{1}{h_\ast - z_0} > 0.
    \]
    Therefore, the whole integral $A'(h_\ast) < 0$.
\end{proof}

In fact, the lemma also implies that $A$ has at most one zero which crossed from positive to negative. The lemma means that if $f$ changes sign exactly once, from positive to negative, then $f \ast \sqrt{\cdot}$ has at most one zero, again crossed from positive to negative. Now we prove the theorem.

\begin{proof}[proof of theorem~\ref{thm:nonmonotone}]
    By the duality relations, it suffices to consider $p \le 0 < q$. Note that the case $q + p = 0$ is already known. We may assume $q + p \neq 0$. Recall that
    \[
    M(x) = \sqrt{\frac{1 - e^{-2x}}{Y(x)}}.
    \]
    Define
    \[
    D(x) := M''(x) + \left( 2 + \frac{4}{e^{2x} - 1} \right) M'(x).
    \]
    By \eqref{eq:logM}, a direct computation yields
    \[
    \frac{D}{M} = 3\left( \frac{Y'}{2Y} - \frac{1}{e^{2x} - 1} \right)^2 + \frac{Y'}{2Y} \left( -\frac{Y''}{Y'} - 2 \right).
    \]
    Assume $D(x_0) = 0$. Since $M > 0$, by lemma~\ref{lem:log-convexity}, we have
    \begin{equation} \label{eq:0-2}
        0 < -\frac{Y''(x_0)}{Y'(x_0)} \le 2.
    \end{equation}
    Further calculation gives
    \[
    \left( \frac{D}{M} \right)'(x_0) = \frac{Y'}{2Y} \left( -\frac{Y''}{Y'} \right)' - \left( 2 + \frac{Y''}{Y'} \right) \left[ \frac{(Y')^2}{4Y^2} + \frac{Y'}{3Y} \left( 1 - \frac{Y''}{Y'}  \right) + \frac{5}{(e^{2x} - 1)^2} \right].
    \]
    Using \eqref{eq:0-2} and lemma~\ref{lem:log-convexity}, we know the first term is strictly negative and the second term is nonpositive. Hence $(D/M)'(x_0) < 0$ and then
    \[
    D'(x_0) < 0.
    \]
    It follows that $D$ has at most one zero. If it exists, it must cross from positive to negative. Now let $h = e^{2E} - 1$. From \eqref{eq:new Theta(E)}, one has
    \begin{align*}
        \Theta'(E) & = \frac{1 + h}{2(q - p)} \int_{0}^{1} \frac{\sqrt{z}}{(1 + hz) \sqrt{1 - z}} M' \left(\frac{1}{2} \log(1 + hz) \right) \dd z \\
        & = \frac{1 + h}{2(q - p)h} \int_{0}^{h} \frac{\sqrt{z}}{(1 + z) \sqrt{h - z}} M' \left(\frac{1}{2} \log (1 + z) \right) \dd z \\
        & = -\frac{2(1 + h)}{(q - p)h^2} \int_{0}^{h} \frac{z^2}{2(1 + z)} M' \left(\frac{1}{2} \log (1 + z) \right) \dd \left( \sqrt{\frac{h - z}{z}} \right) \\
        & = \frac{1 + h}{2(q - p)h^2} \int_{0}^{h} \frac{z^{3/2}}{(1 + z)^2} D \left( \frac{1}{2} \log(1 + z) \right) \sqrt{h - z} \dd z,
    \end{align*}
    where the last equality follows from integration by parts. The boundary terms vanish and the integration by parts is justified by the expansion of $M$ at $0$. The factor on the right-hand side, apart from $\sqrt{h - z}$, exhibits the same sign changes as $D$.
    
    Therefore, if $D$ does not change sign, then neither does $\Theta'$, which implies that $\Theta$ has no critical points. If $D$ changes sign once from positive to negative, it follows from lemma~\ref{lem:sign} that $\Theta$ has at most one critical point, which must be a nondegenerate local maximum.
\end{proof}

\section{Proof of theorem~\ref{thm:main}} \label{sec:4}

The preceding discussion completely clarifies the profile of $\Theta(p, q, r)$ as a function of $r$ across the entire $(p, q)$-plane: away from $(1, 2), (-2, -1), (-2, 2)$, it is either strictly increasing, strictly decreasing or strictly increasing and then strictly decreasing. Consequently, only an upper bound estimate is needed to fully determine the range of $\Theta$, thereby yielding a complete characterization of the exact number of nonconstant solutions. We now establish this upper bound estimate.

\begin{lemma} \label{lem:bound}
    Suppose
    \[
    q - p > 1, \quad q \ge 2p, \quad 2q \ge p.
    \]
    Then
    \[
    \Theta(p, q, r) < \pi, \quad r > 1.
    \]
\end{lemma}

\begin{proof}
    By duality \eqref{eq:dual3}, we may assume $q + p \ge 0$. Note that
    \[
    \Theta(-2, 2, r) \equiv \frac{\pi}{2}, \quad \Theta(-2, -1, r) = \Theta(1, 2, r) \equiv \pi.
    \]

    If $p < -2$, by $q + p \ge 0$, we have
    \[
    q \ge -p > 2.
    \]
    By \ref{lem:pqmonotone}, we obtain
    \[
    \Theta(p, q, r) < \Theta(-2, 2, r) = \frac{\pi}{2} < \pi.
    \]

    If $-2 \le p \le 1$, consider the open segment $q - p = 1$. It lies in the strictly decreasing region. By \eqref{eq:limit1}, we have
    \[
    \Theta(p, q, r) < \Theta(p, p + 1, r) \le \Theta(p, p + 1, 1+) = \pi.
    \]

    If $p > 1$, consider the line $q = 2p$. It lies in the strictly increasing region. By \eqref{eq:limit2}, we have
    \[
    \Theta(p, q, r) \le \Theta(p, 2p, r) < \Theta(p, 2p, \infty) = \frac{2p}{2(2p - p)} \pi = \pi.
    \]
\end{proof}

We are now finally in a position to prove our main theorem. Some of these results have already been established by Li--Wan\cite{LW26}. For the sake of completeness, we will reprove some of them.

\begin{proof}[proof of theorem~\ref{thm:main}]
    At the three exceptional points, the solutions were written out explicitly by Li-Wan, which directly implies that $N(p, q) = \infty$.

    If $(p, q) = (1, 2)$, then up to a rotation,
    \[
    u(\theta) = 1 + \lambda \cos \theta, \quad 0 \le \lambda < 1.
    \]

    If $(p, q) = (-2, -1)$, then up to a rotation,
    \[
    u(\theta) = \frac{\sqrt{1 - \mu^2 \sin^2 \theta} - \mu \cos \theta}{1 - \mu^2}, \quad 0 \le \mu < 1.
    \]

    If $(p, q) = (-2, 2)$, then up to a rotation,
    \[
    u(\theta) = \sqrt{\lambda^2 \cos^2 \theta + \lambda^{-2} \sin^2 \theta}, \quad \lambda > 0.
    \]

    For $q - p \le 0$, there are only constant solutions, hence $N(p, q) = 0$.

    For $0 < q - p < 1$, we first assume $q > 2p$, $2q > p$. This lies in the strictly decreasing region. By \eqref{eq:limit1} and \eqref{eq:limit2}, we have
    \[
    \Theta(p, q, 1+) = \frac{\pi}{\sqrt{q - p}} > \pi, \quad \frac{\pi}{2} \le \Theta(p, q, \infty) < \pi. 
    \]
    Hence $N(p, q) = 1$. If $q \le 2p$ or $2q \le p$, we still have $\Theta(p, q, 1+) > \pi$ but $\Theta(p, q, \infty) \ge \pi$. By theorem~\ref{thm:nonmonotone}, we have $\Theta(p, q, r) > \pi$ for all $r > 1$. Hence $N(p, q) = 0$.

    For $q - p = 1$, if $-2 < p < 1$, this lies in the strictly decreasing region with $\Theta(p, p + 1, 1+) = \pi$ and $\Theta(p, p + 1, \infty) \ge \pi/2$. Hence $N(p, q) = 0$. If $p > 1$ or $p < -2$, this lies in the strictly increasing region. Then $\Theta(p, p + 1, r) > \Theta(p, p + 1, 1+) = \pi$. Hence $N(p, q) = 0$.

    For $1 < q - p \le 4$, we first assume $q \ge 2p$, $2q \ge p$. we have
    \[
    \frac{\pi}{2} \le \Theta(p, q, 1+) = \frac{\pi}{\sqrt{q - p}} < \pi, \quad \Theta(p, q, \infty) \ge \frac{\pi}{2}.
    \]
    By lemma~\ref{lem:bound} and theorem~\ref{thm:nonmonotone}, we have $N(p, q) = 0$. If $q < 2p$, this lies in the strictly increasing region with
    \[
    \frac{\pi}{2} \le \Theta(p, q, 1+) = \frac{\pi}{\sqrt{q - p}} < \pi, \quad \Theta(p, q, \infty) > \pi.
    \]
    Then $N(p, q) = 1$. The case $2q < p$ is the same as $q < 2p$ by \eqref{eq:dual3}.

    For $q - p > 4$, we first assume $q < 2p$. This lies in the strictly increasing region with
    \[
    \Theta(p, q, 1+) = \frac{\pi}{\sqrt{q - p}} < \frac{\pi}{2}, \quad \Theta(p, q, \infty) > \pi.
    \]
    Hence $N(p, q) = \ceil{\sqrt{q - p}} - 1$. The case $2q < p$ is the same as $q < 2p$ by \eqref{eq:dual3}. If $p < 0 < q$ and $1/(-p) + 1/q \le 1$. This lies in the strictly increasing region with
    \[
    \Theta(p, q, 1+) = \frac{\pi}{\sqrt{q - p}} < \frac{\pi}{2}, \quad \Theta(p, q, \infty) = \frac{\pi}{2}.
    \]
    Hence $N(p, q) = \ceil{\sqrt{q - p}} - 3$. The remaining regions satisfy $q \ge 2p$, $2q \ge p$ with
    \[
    \Theta(p, q, 1+) = \frac{\pi}{\sqrt{q - p}} < \frac{\pi}{2}, \quad \Theta(p, q, \infty) \ge \frac{\pi}{2}.
    \]
    By lemma~\ref{lem:bound} and theorem~\ref{thm:nonmonotone}, we have $N(p, q) = \ceil{\sqrt{q - p}} - 2$.

    We complete the proof.
\end{proof}

\section*{Declaration on the use of AI}

The author used GPT--6 Astra, developed by OpenAI. It assisted in verifying the asymptotic expansions of the relevant functions, thereby checking the well-posedness of corresponding analytic operations. Guided by author's analytic directions, it played a substantial role in discovering useful differential identities involving $H$ and $Y$, and introducing the auxiliary function $D$ into the representation of $\Theta'(E)$. The author takes full responsibility for its mathematical content.

\bibliographystyle{plain}

\bibliography{reference}

\end{document}